\documentclass[12pt]{amsart}
\pdfoutput=1
\usepackage{mathrsfs}
\usepackage{fullpage}
\usepackage{tikz}
\usepackage{color,url}
\usepackage{hyperref}

\usepackage{amsfonts,amssymb,amscd,amsmath,enumerate,verbatim}
\usepackage[normalem]{ulem}
\usepackage{mwe}
\usepackage{wrapfig}
\usepackage{pdflscape}
\usepackage{subcaption}

\newtheorem{theorem}{Theorem}[section]
\newtheorem{lemma}[theorem]{Lemma}
\newtheorem{corollary}[theorem]{Corollary}
\newtheorem*{corollary*}{Corollary}
\newtheorem{proposition}[theorem]{Proposition}
\newtheorem{example}[theorem]{Example}
\newtheorem*{conjecture}{Conjecture}
\newtheorem{remark}{Remark}
\newtheorem{definition}[theorem]{Definition}

\newcommand{\bs}{\backslash}

\def\N{\mathbb{N}}

\def\Z{\mathbb{Z}}

\newcommand{\la}{\langle}
\newcommand{\ra}{\rangle}
\newcommand{\R}{\mathbb R}
\def\mC{\mathcal C}
\newcommand{\kwpqd}{KW_D(p,q)}
\renewcommand{\k}{\Bbbk}

\usepackage{scalerel}
\newcommand{\precdot}{\mathrel{\prec\mathrel{\mkern-8mu}\scalerel*{\cdot}{\prec}}}

\DeclareMathOperator{\ap}{Ap}

\title{The Betti Numbers of Kunz-Waldi Semigroups}
      
\author[M. Gonz\'alez-S\'anchez]{Mario Gonz\'alez-S\'anchez} \address{Instituto de Investigaci\'on en Matem\'aticas de la Universidad de Valladolid (IMUVA), Universidad de Valladolid, 47011 Valladolid, Spain.}
\email{mario.gonzalez.sanchez@uva.es}

\author[S. Singh]{Srishti Singh}\address{Department of Mathematics, University of Tennessee, Knoxville TN 37996, USA.}
\email{ssingh63@tennessee.edu}

\author[H. Srinivasan]{Hema Srinivasan}
\address{Department of Mathematics, University of Missouri, Columbia MO 65211, USA.}
\email{srinivasanh@missouri.edu}

\thanks{The first author has been supported in part by the grant PID2022-137283NB-C22 funded by MICIU/AEI/ 10.13039/501100011033 and by ERDF/EU, and thanks financial support from European Social Fund, {\it Programa Operativo de Castilla y Le\'on}, and {\it Consejer\'ia de Educaci\'on de la Junta de Castilla y Le\'on.}\\This collaboration is the outcome of the first author's visit to the University of Missouri during Fall 2024.
}
\thanks{The third author is supported by a grant from Simons Foundation.}


\begin{document}

\begin{abstract}
Given two coprime numbers $p<q$, KW semigroups contain $p,q$ and are contained in $\la p,q,r \ra$ where $2r= p,q, p+q$ whichever is even.  These semigroups were first introduced by Kunz and Waldi.  Kunz and Waldi proved that all $KW$ semigroups of embedding dimension $n\geq 4$ have Cohen-Macaulay type $n-1$ and first Betti number ${n \choose 2}$.  In this paper, we characterize KW semigroups whose defining ideal is generated by the $2\times 2$ minors of a $2\times n$ matrix. In addition, we identify all KW semigroups that lie on the interior of the same face of the Kunz cone $\mC_p$ as a KW semigroup with determinantal defining ideal. Thus, we provide an explicit formula for the Betti numbers of all those KW semigroups. 

\end{abstract}

\maketitle

\section*{Introduction}

Let $3\leq p < q$ be two relatively prime integers. In \cite{Kunz2014}, Kunz and Waldi build numerical semigroups with multiplicity $p$ by filling in the gaps of the symmetric semigroup $\la p,q \ra$ in a systematic way. In particular, they characterize all numerical semigroups $\la p,q \ra \subsetneq H \subset \la p,q,r \ra$, where $r=p/2$ if p is even, $r=q/2$ if $q$ is even, and $r = (p+q)/2$ otherwise. Such semigroups $H$ are in one-to-one correspondence to the lattice paths, with right and downward steps, in the rectangle $R \subset \R^2$ with vertices $(0,0)$, $(0,p'-1)$, $(q'-1,p'-1)$, and $(q'-1,0)$, where $p' = \lfloor p/2 \rfloor$ and $q' = \lfloor q/2 \rfloor$. We call these semigroups Kunz-Waldi semigroups and denote by $KW(p,q)$ the class formed by all of them. In fact, a semigroup $H = \la p,q,h_1,\ldots,h_{n-2} \ra \in KW(p,q)$ if and only if $h_i = pq-x_ip-y_iq$ satisfying $0 < x_1 < \dots < x_{n-2} \leq q/2$ and $p/2 \geq y_1 > \dots > y_{n-2} > 0$, as observed in Remark 1 of \cite{Singh2024}, and the embedding dimension of $H$ is $e(H) = n$. \newline

Let $H = \la p,q,h_1,\ldots,h_{n-2} \ra \in KW(p,q)$ and fix a field $\k$. Consider the polynomial ring $T= \k[u,v,u_1,...,u_{n-2}]$ graded by $H$ via $\deg(u)=p$, $\deg(v) = q$ and $\deg(u_i) = h_i$ for all $1\leq i \leq n-2$, and let $t$ be another unknown. Define $\phi_H : T \to \k[t]$ by $u \mapsto t^{p}$, $v \mapsto t^q$, $u_i \mapsto t^{h_i}$, $1 \leq i \leq n-2$. The {\it defining ideal of $H$} is $I_H =\ker \phi_H$, and it is homogeneous, prime, and binomial. In line with the convention of calling a numerical semigroup $H$ Gorenstein or Complete Intersection when $I_H$ is Gorenstein or Complete Intersection, we say $H$ is determinantal if the ideal $I_H$ is a determinantal ideal.  We know from the theorem in Appendix A of \cite{KUNZ2017397} that $I_H$ is minimally generated by the following ${n \choose 2 }$ homogeneous binomials:
    \begin{align*}
        f_{ij} &= u_iu_j-u^{q-x_i-x_j}v^{p-y_i-y_j}, \quad 1 \leq i \leq j \leq n-2 \\
        g_i &= v^{y_i-y_{i+1}}u_i-u^{x_{i+1}-x_i}u_{i+1}, \quad 1 \leq i \leq n-3 \\
        \eta_1 &= v^{p-y_1}-u^{x_1}u_1 \\
        \eta_2 &= v^{y_{n-2}}u_{n-2}-u^{q-x_{n-2}}.
    \end{align*}
Since $I_H$ is homogeneous, the semigroup algebra $\k[H] \simeq T/I_H$ is a $H$-graded module, and we can consider a minimal graded free resolution of $T/I_H$,
\[0 \to T^{\beta_{n-1}} \to T^{\beta_{n-2}} \to \dots \to T^{\beta_1} \to T \to T/I_H \to 0 \, .\]
The numbers $\beta_i$, $0\leq i\leq n-1$, are the {\it Betti numbers} of $T/I_H$. Note that $\beta_0=1$, $\beta_1 = {n \choose 2}$ is the number of elements in any minimal generating set of $I_H$ formed by homogeneous polynomials, and $\beta_{n-1} = n-1$ by Corollary 3.1 of \cite{KUNZ2017397}. \newline

Given a numerical semigroup $S \subset \N$, let $m = \min\left( S\bs \{0\} \right)$ be its multiplicity, and write 
\[\ap(S) = \{s \in S \mid s-m \notin S\} = \{0,a_1,\ldots,a_{m-1} \}\]
for the {\it Apéry set} of $S$ with respect to $m$, where for all $1\leq i \leq m-1$, $a_i\in S$ is the minimum element in $S$ congruent to $i$ modulo $m$. The {\it Apéry coordinate vector of $S$} with respect to $m$ is the tuple $(a_1,\ldots,a_{m-1})$. For each $m\in \Z_{>0}$, $m\geq 2$, the {\it Kunz cone} $\mC_m \subset \R_{\geq 0}^{m-1}$ is a pointed cone with defining inequalities $a_i +a_j \geq a_{i+j}$ whenever $i +j \neq 0$, where $i,j \in \Z_m \bs \{0\}$. If $S_1,S_2$ are two numerical semigroups of multiplicity $m$ whose Apéry coordinate vectors lie on the interior of the same face of the Kunz cone $\mC_m$, then they have the same Betti numbers by Theorem 2.7 of \cite{Braun2023}. \newline

The study of determinantal ideals is a well-developed area in commutative algebra, as these ideals provide extensive information about the structure of the ring. In particular, their minimal free resolutions are completely understood. In \cite{Goto2018} and \cite{VanKien2020}, the authors give criteria for a numerical semigroup of embedding dimension $n$ to have the defining ideal generated by the $2\times 2$ minors of a $2\times n$ matrix in terms of Pseudo-Frobenius numbers, in some specific cases. In \cite{Singh2024}, the authors examine the defining ideals of KW semigroups and show that the ideal can be written as a sum of determinantal ideals (Theorem $4.1$, \cite{Singh2024}). However, the results did not extend beyond this, and a key motivation was to determine whether any KW semigroups are determinantal.

In this paper, we characterize the semigroups $H \in KW(p,q)$ whose defining ideal is minimally generated by the $2\times 2$ minors of a $2\times n$ matrix, in terms of the Pseudo-Frobenius numbers of $H$ (Proposition \ref{prop:determinantal}), and later we use the Kunz cone $C_p$ to compute the Betti numbers of a larger class. These determinantal semigroups from a subclass $\kwpqd$ of $KW(p,q)$, defined as follows: a semigroup $H = \la p,q,h_1,\ldots,h_{n-2} \ra \in KW(p,q)$ is in $\kwpqd$ if  there exist $x,y\in \Z_{>0}$ with $(n-2)x \leq q/2$, $(n-2)y \leq p/2$, and $x_i = ix$, $y_i = (n-1-i)y$ for all $1\leq i \leq n-2$. In Proposition \ref{prop:determinantal} we prove that the $2\times 2$ minors of the matrix
\[\begin{bmatrix}
u_{n-2} & u^x & v^{p-(n-1)y} & u_1 & u_2 & \cdots & u_{n-4} & u_{n-3} \\
u^{q-(n-1)x} & v^y & u_1 & u_2 & u_3 & \cdots &  u_{n-3} & u_{n-2}
\end{bmatrix} \]
generate the ideal $I_H$ for a semigroup $H \in \kwpqd$, and in Theorem \ref{thm:determinantal} we prove that the semigroups in $\kwpqd$ are the only ones in $KW(p,q)$ whose defining ideal is determinantal. 
For any semigroup $H \in \kwpqd$, the Eagon-Northcott complex \cite{Eagon1962} resolves the ideal $I_H$. Hence, their Betti numbers are given by $\beta_i = i {n \choose i+1}$ for all $1\leq i \leq n-1$. Moreover, for two semigroups $H,G \in KW(p,q)$, we characterize when $G$ and $H$ lie on the interior of the same face of the Kunz cone $\mC_p$ in Proposition \ref{prop:KWfaces}. As a consequence, we obtain a larger subclass of semigroups of $KW(p,q)$ whose Betti numbers are given by $\beta_i = i {n \choose i+1}$, $1\leq i \leq n-1$.

\section{When is a $KW(p,q)$- semigroup determinantal?}

The goal of this section of to completely characterize all Kunz-Waldi numerical semigruops whose defining ideals are determinantal. To this end, we give a necessary and sufficient condition for any $H \in KW (p,q)$ to have a determinantal defining ideal $I_H$ (\ref{thm:determinantal}) and provide a matrix $A$ such that $I_H = I_{2 \times 2}(A)$ (\ref{prop:determinantal}), i.e. $I_H$ is generated by the $2\times 2$ minors of $A$.

\begin{definition}\label{def:determinantal}
Let $3 \leq p < q$ be relatively prime. We define a subclass $\kwpqd$ of $KW(p,q)$ as follows. A semigroup $H = \la p,q,h_1,\ldots,h_{n-2} \ra \in KW(p,q)$ is in $\kwpqd$ if there exist $x,y \in \Z_{>0}$ such that $(n-2)x \leq q/2$, $(n-2)y \leq p/2$, and $h_i = pq-x_ip-y_iq$, where $x_i = ix$ and $y_i = (n-1-i)y$ for all $1\leq i \leq n-2$.
\end{definition}

The notation chosen to denote this class will be justified after Theorem \ref{thm:determinantal}, where we prove that the semigroups in $\kwpqd$ are precisely all the semigroups in $KW(p,q)$ with determinantal defining ideal.

\begin{remark}
Let $H \in \kwpqd$ for some $x,y \in \Z_{>0}$. Note that $h_{i+1}-h_i = yq-xp$ for all $1\leq i \leq n-3$. Thus, the generators $h_1,h_2,\ldots,h_{n-2}$ necessarily form an arithmetic sequence with common difference $yq-xp$. Note that this sequence is increasing if $yq>xp$ and decreasing otherwise.
\end{remark}

    \begin{proposition}\label{prop:determinantal}
        If $H \in \kwpqd$ for some $x,y,n\in \Z_{>0}$
        then $I_H$ is generated by the $2 \times 2$ minors of the following $2 \times n$ matrix $$ A=
        \begin{bmatrix}
            u_{n-2} & u^x & v^{p-(n-1)y} & u_1 & u_2 & \cdots & u_{n-4} & u_{n-3} \\
            u^{q-(n-1)x} & v^y & u_1 & u_2 & u_3 & \cdots &  u_{n-3} & u_{n-2}
        \end{bmatrix} \, .$$
 \end{proposition}

\begin{proof}
Let $a_{ij}$ denote the determinant of the submatrix of $A$ obtained from the $i$- and $j$-th columns, and $I_2(A)$ be the ideal generated by all the $a_{ij}$. Each $a_{ij}$ is a $H$-homogeneous binomial and thus, is in $I_H$. Therefore, $I_2(A) \subset I_H$. Let us prove the other inclusion. 
Note that for all $1\leq i \leq n-3$, $g_i = -a_{2,i+3}$; $\eta_1 = -a_{23}$, and $\eta_2 = a_{12}$, so all of them are in $I_2(A)$. To finish the proof, it suffices to show that $f_{ij} \in I_2(A)$ for all $1\leq i\leq j \leq n-2$. We have $f_{1,n-2} = a_{13} \in I_2(A)$ and for all $1\leq j \leq n-3$, one can check that
\[f_{1,j} = -a_{3,j+3} - \sum_{k=0}^{n-j-4} \left( u^{kx} v^{p-(n+k)y} a_{2,j+4+k} \right) + u^{(n-j-3)x} v^{p-(2n-j-3)y} a_{12} \in I_2(A) \, .\]
Moreover, for all $2\leq i \leq n-2$,
\[f_{i,n-2} = a_{1,i+2} + \sum_{k=0}^{i-2} \left( u^{q-(n+k)x}v^{ky} a_{2,i+1-k} \right) \in I_2(A) \, .\]
Finally, note that for all $2\leq i \leq j \leq n-3$, $f_{i,j} = f_{i-1,j+1} - a_{i+2,j+3}$, and hence all the $f_{ij}$ belong to $I_2(A)$.
\end{proof}

When Kunz and Waldi initially construct these numerical semigroups of multiplicity $p$ obtained by adjoining arbitrary gaps of $\la p,q \ra$, no conditions are imposed on the gaps. In this initial work, they discuss the set of Pseudo-Frobenius elements of these semigroups (page 668, \cite{Kunz2014}). Later, they prove that if the gaps are chosen so that the resulting numerical semigroups belong to $KW(p,q)$, then their type is one less than the embedding dimension  (Example 4.6, \cite{KUNZ2017397}). From this, one can deduce the complete set of Pseudo-Frobenius elements for any $H \in KW(p,q)$, as we will illustrate below.

\begin{lemma} \label{lemma:PF(H)}
    Let $H \in KW(p,q)$ of embedding dimension $n$. Then $PF(H)=\{g_i:=pq-(x_i+1)p-(y_{i+1}+1)q \mid 0 \leq i \leq n-2\}$, where $x_0=y_{n-1}=0$.
\end{lemma}
\begin{proof}

Set $x_0=y_{n-1}=0$ and let  $\mathcal H = \{g_i:=pq-(x_i+1)p-(y_{i+1}+1)q \mid 0 \leq i \leq n-2\}$. Recall from the introduction that $H= \la p,q, h_1,...,h_{n-2} \ra$ is in one-to-one correspondence to a lattice path $l$ in the rectangle $R$. In particular, each $h_i$ corresponds to the point $(x_i-1,y_i-1)$ under the line $pq-p-q=px+qy$ (\cite{Kunz2014}). To see that $\mathcal H = PF(H)$, first note that $\mathcal{H} = \{h\in \N \setminus H \mid h+p \in H, h+q \in H\}$. This is clear from Figure \ref{fig:pseudofrob}. Hence, $PF(H) \subset \mathcal{H}$, and the equality $PF(H) = \mathcal{H}$ follows from the fact $|\mathcal{H}| = n-1 = |PF(H)|$ (the type of $H$ is $n-1$).
\end{proof}

\begin{figure}[htbp]
\centering
\includegraphics[width=1\textwidth]{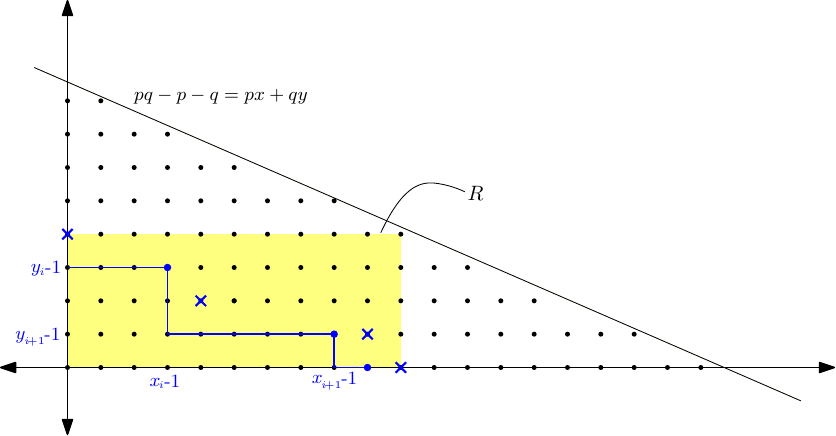}
\caption{Pseudo Frobenius Elements of $H \in KW(p,q)$ correspond to the points $\times$.}
\label{fig:pseudofrob}
\end{figure}

\begin{theorem}\label{thm:determinantal}
Let $H \in KW(p,q)$ of embedding dimension $n$. Then the following are equivalent:
\begin{enumerate}
    \item\label{thm:determinantal_1} $H \in \kwpqd$.
    \item\label{thm:determinantal_2} There exist positive integers $x,y$ such that $(n-2)y \leq p/2$, $(n-2)x \leq q/2$ and $I_H$ is generated by the $2\times 2$ minors of the matrix \[\begin{bmatrix}
            u_{n-2} & u^x & v^{p-(n-1)y} & u_1 & u_2 & \cdots & u_{n-4} & u_{n-3} \\
            u^{q-(n-1)x} & v^y & u_1 & u_2 & u_3 & \cdots &  u_{n-3} & u_{n-2}
        \end{bmatrix} \, .\]
    \item\label{thm:determinantal_3} There exist $H$-homogeneous polynomials $F_i,G_i \in \left( u,v,u_1,\ldots,u_{n-2} \right)$ such that $I_H$ is generated by the $2\times 2$ minors of the matrix \[\begin{bmatrix}
            F_1 & F_2 & \dots & F_n \\
            G_1 & G_2 & \dots & G_n
        \end{bmatrix} \, .\]
    \item\label{thm:determinantal_4} ${\rm PF}(H) = \{z+k,z+2k,\ldots,z+(n-1)k\}$ for some $z \geq 0$ and $k >0$.
\end{enumerate}
When this is the case, ${\rm PF}(H) = \{pq-(ix+1)p - \left( (n-i-2)y+1 \right) q \mid 0\leq i \leq n-2\}$, for the numbers $x,y\in \Z_{>0}$ defining $H \in \kwpqd$, and hence $k = |xp-yq|$.
\end{theorem}

\begin{proof}
\eqref{thm:determinantal_1} $\Rightarrow$ \eqref{thm:determinantal_2} is Proposition \ref{prop:determinantal}, \eqref{thm:determinantal_2} $\Rightarrow$ \eqref{thm:determinantal_3} is trivial and \eqref{thm:determinantal_3} $\Rightarrow$ \eqref{thm:determinantal_4} is proved in general for any numerical semigroup in Section 2 of \cite{VanKien2020}.

\eqref{thm:determinantal_4} $\Rightarrow$ \eqref{thm:determinantal_1} By Lemma \ref{lemma:PF(H)}, $PF(H)=\{g_i:=pq-(x_i+1)p-(y_{i+1}+1)q \mid 0 \leq i \leq n-2\}$, where $x_0=y_{n-1}=0$.
Since $PF(H)$ satisfies (4) by our hypothesis, we must have that the difference between any two consecutive elements is constant. Firstly, for any $0 \leq i \leq n-2$, 
\[\begin{split}
g_i-g_{i+1} &= pq-(x_i+1)p-(y_{i+1}+1)q  - [pq-(x_{i+1}+1)p-(y_{i+2}+1)q ] \\
&= p(x_{i+1}-x_i) +q (y_{i+2}-y_{i+1}) \, .
\end{split}\]
Let $\alpha_i = x_{i+1}-x_i$ and $\beta_i = y_{i+2}-y_{i+1}$. Now we must have that for any $0 \leq i <j \leq n-3$, $$p \alpha_i +q \beta_i = p \alpha_j + q \beta_j \implies p(\alpha_i - \alpha_j) = q (\beta_j -\beta_i).$$

Since $p$ and $q$ are relatively prime, there exists $l \in \Z$ such that $\alpha_i-\alpha_j = ql$. But $|\alpha_i-\alpha_j| = |x_{i+1}+x_j-x_i-x_{j+1}| \leq q-2$, so $l=0$. Thus, since $x_{i+1}>x_i$, there is some $x \in \Z_{>0}$ such that  $\alpha_i = \alpha_j = x$.  The recursive definition of $\alpha_i$ now gives us $x_i = ix$ for $1 \leq i \leq n-2$. If $x > \frac{q}{2(n-2)}$, then $x_{n-2} = (n-2)x > q/2$, a contradiction.

Similarly, $\beta_i = \beta_j = -y$, for some $y \in \Z_{>0}$ since $y_{i+2}< y_{i+1}$. This implies $y_{n-2}=y, y_{n-3}=2y,...,y_1=(n-2)y$. If $y > \frac{p}{2(n-2)}$ then $y_1 >p/2$, a contradiction.

We have shown that $H$ must be as in Definition \ref{def:determinantal}. In addition, when \eqref{thm:determinantal_4} holds, we have shown that ${\rm PF}(H) = \{pq-(ix+1)p - \left( (n-i-2)y+1 \right) q \mid 0\leq i \leq n-2\}$, which implies $k=|xp-yq|$.
\end{proof}

\begin{remark}
Let $3\leq p<q$ relatively prime integers and denote $p' = \lfloor p/2\rfloor$, and $q' = \lfloor q/2 \rfloor$. 
By Theorem \ref{thm:determinantal}, the cardinality of $\kwpqd$ is
\[|\kwpqd| = \sum_{n=3}^{p'+2} \left\lfloor \frac{p'}{n-2}\right\rfloor \left\lfloor \frac{q'}{n-2}\right\rfloor = \sum_{n=1}^{p'} \left\lfloor \frac{p'}{n}\right\rfloor \left\lfloor \frac{q'}{n}\right\rfloor \, ,\]
and the cardinality of $KW(p,q)$ is
\[|KW(p,q)| = \sum_{n=3}^{p'+2} {p' \choose n-2} {q' \choose n-2} = \sum_{n=1}^{p'} {p' \choose n}{q' \choose n} = {p'+q' \choose p'}-1 \, .\]
Thus, the proportion of semigroups in $KW(p,q)$ whose defining ideal is determinantal is 
\[\rho_D(p,q):= \frac{|\kwpqd|}{|KW(p,q)|} = \frac{\sum_{n=1}^{p'} \left\lfloor \frac{p'}{n}\right\rfloor \left\lfloor \frac{q'}{n}\right\rfloor}{{p'+q' \choose p'}-1} \, .\]
\end{remark}

\section{Betti Sequence of any $H \in KW(p,q)$}

Since the ideal $I_H \subset T = \k[u,v,u_1,\ldots,u_{n-2}]$ for any $H \in \kwpqd$ is generated by the $2\times 2$ minors of a $2\times n$ matrix and its height is $n-1$ $(=n-2+1)$, it is resolved by the Eagon-Northcott complex (Theorem 2 of \cite{Eagon1962}). In particular, if the minimal graded free resolution of $T/I_H$ is \[0 \to T^{\beta_{n-1}} \to T^{\beta_{n-2}} \to \dots \to T^{\beta_1} \to T \to T/I_H \to 0 \, ,\]
then the Betti numbers are given by the formula $\beta_i = i {n \choose i+1}$ for all $1\leq i \leq n-1$, by the construction of the Eagon-Northcott complex.
We want to expand this to compute the Betti sequence of more KW-semigroups. One way to do this is by using the fact that any two numerical semigroups of multiplicity $p$ lying in the interior of the same face of the Kunz cone $\mC_p$ have the same Betti sequence. \newline

In this section, we identify all $H \in KW(p,q)$ that lie on the interior of the same face of $\mC_p$ as some $H' \in \kwpqd$. To achieve this, we compute the Ap\'ery posets of $H \in KW(p,q)$ and use the fact that two semigroups lie on the same face of $\mC_p$ if and only if they have the same Ap\'ery posets.

We first recall how to construct the Apéry set and poset of any semigroup $H \in KW(p,q)$ from the lattice path defining $H$. For the rest of this section, assume that $n\geq 4$.

\begin{proposition}\label{prop:ApH}
Let $H = \langle p,q,h_1,\ldots,h_{n-2} \rangle \in KW(p,q)$ be the semigroup defined by the sequences $x_1<x_2<\dots < x_{n-2}$ and $y_1>y_2>\dots > y_{n-2}$, and set $y_{n-1} := 0$. The Ap\'ery set of $H$ is
\[{\rm Ap}(H) = \{\lambda q \mid 0\leq \lambda < p-y_1 \} \cup \left( \cup_{i=1}^{n-2} \{h_i +\lambda q \mid 0\leq \lambda < y_i-y_{i+1}\} \right) \, .\]
\end{proposition}

\begin{proof}
Let $A$ be the Ap\'ery set of the numerical semigroup $\langle p,q \rangle$, $A = \{0,q,2q,\ldots,(p-1)q\}$. We can obtain the Ap\'ery set ${\rm Ap}(H)$ from $A$ as follows. 
Recall that 
\[H = \langle p,q \rangle \cup \{pq-xp-yq \mid x\leq x_i, y \leq y_i \text{ for some } 1\leq i \leq n-2 \} \, .\]
Note that $pq-xp-yq \equiv -yq \pmod{p}$, so the elements that we have to replace in $A$ are the ones congruent to $-yq$ modulo $p$ for $0<y\leq y_1$. For each one of these congruence classes, we choose the smallest element in $H$, i.e. the one with the largest $x$. This corresponds to the element in the lattice path defining $H$ that has $Y$-coordinate $y-1$.
Therefore, 
\[{\rm Ap}(H) = \{\lambda q \mid 0\leq \lambda < p-y_1 \} \cup \left( \cup_{i=1}^{n-3} \{h_i +\lambda q \mid 0\leq \lambda < y_i-y_{i+1}\} \right) \cup \{h_{n-2}+\lambda q \mid 0 \leq \lambda < y_{n-2}\} \, .\]
\end{proof}

\begin{remark}
Recall that $\operatorname{Ap}(H)$ is homogeneous if, for every
$z \in \operatorname{Ap}(H)$, all the factorizations of $z$ have the same length.
By Corollary~3.10 of \cite{JZ}, this holds whenever $I_H$ has a minimal generating
set of binomials such that $u$ appears in each of its elements that is not
homogeneous with respect to the standard grading.

Among the minimal generators of $I_H$ listed in the introduction, the only one
that may fail to involve $u$ is
\[
  f_{n-2,n-2} = u_{n-2}^2 - u^{\,q-2x_{n-2}}\, v^{\,p-2y_{n-2}},
\]
and this happens if and only if $q$ is even and $x_{n-2} = q/2$. Hence, if $q$ is
odd or $x_{n-2} < q/2$, then $\operatorname{Ap}(H)$ is homogeneous.

Assume now that $q$ is even and $x_{n-2} = q/2$. Then
$f_{n-2,n-2} = u_{n-2}^2 - v^{\,p-2y_{n-2}}$, which corresponds to the relation
\begin{equation}\label{eq:star}
  2h_{n-2} = (p-2y_{n-2})\,q ,
  \tag{$*$}
\end{equation}
and it is not homogeneous because $p$ is odd. We distinguish two cases:
\begin{itemize}
  \item If $y_1 \geq 2y_{n-2}$, then
  \[
    f' := f_{n-2,n-2} + v^{\,y_1-2y_{n-2}}\,\eta_1
        = u_{n-2}^2 - u^{x_1} u_1\, v^{\,y_1-2y_{n-2}} \in I_H ,
  \]
  and replacing $f_{n-2,n-2}$ by $f'$ we obtain another minimal generating set of
  $I_H$ in which $u$ appears in every element. Thus, $\operatorname{Ap}(H)$ is
  homogeneous.
  \item If $y_1 < 2y_{n-2}$, then $p-2y_{n-2} < p-y_1$, so
  $(p-2y_{n-2})\,q \in \operatorname{Ap}(H)$ by Proposition~2.1. By \eqref{eq:star}, this
  element has factorizations of lengths $p-2y_{n-2}$ and $2$, which are different
  since $p$ is odd. Thus, $\operatorname{Ap}(H)$ is not homogeneous.
\end{itemize}

Therefore, $\operatorname{Ap}(H)$ is homogeneous unless $q$ is even,
$x_{n-2} = q/2$ and $y_1 < 2y_{n-2}$.
\end{remark}

We now consider a poset structure on $\ap(H)$.

\begin{definition}
Let $S$ be a numerical semigroup with multiplicity $m$, and write $\ap (S) = \{a_0:=0,a_1,...,a_{m-1}\}$ so that $a_i \equiv i \pmod{m}$ for all $i$. The \textit{Ap\'ery poset} of $S$ is a poset $\mathcal{P}(S) = (\Z_m , \preceq)$, where $i \preceq j$ if and only if $a_j-a_i \in S$ for $i,j \in \Z_m $. 
We write $i \precdot j$ and say $j$ covers $i$ if $i \prec j$ and there is no $k$ such that $i \prec k \prec j$. 
\end{definition}

\begin{proposition} \label{prop:cover}
For all $i,j \in \Z_m$, $i \precdot j$ if and only if $a_j-a_i$ is a minimal generator of $S$.   
\end{proposition}

\begin{proof}
Being $(\Leftarrow)$ trivial, let us prove $(\Rightarrow)$. Let $i,j \in \Z_m$ such that $i \precdot j$ and write $a_j = a_i+\alpha+\beta$ for some $\alpha,\beta \in H$ with $\alpha$ a minimal generator of $H$. Note that $a_i \prec a_i+\alpha \preceq a_j$. Since $a_j \in \ap(H)$, then $a_i+\alpha \in \ap(H)$, so $a_j=a_i+\alpha$ as $j$ covers $i$.
\end{proof}

 \begin{theorem} (Theorem 3.10, \cite{BGSOW}) \label{thm:KunzApery}
     Any two numerical semigroups with the same multiplicity $m$ belong to the interior of the same face of $\mathcal C_m$ if and only if their Ap\'ery posets are the same.
 \end{theorem}

\begin{example} \label{ex:Apery}
    Consider the semigroups $H= \la 8,17, 60,69,78 \ra $ and $G = \la 8,17, 53,62,55 \ra$ with Ap\'ery sets $\ap (H)  = \{0, 17,34,51,60,69,78,95\}$ and $\ap (G) = \{0, 17,34,51,68,53,62,55\}$. Their Ap\'ery posets are shown in Figure \ref{fig:Aperyposets}.
\end{example}

\begin{figure}[htbp]
\begin{subfigure}[b]{0.45\textwidth}
\centering
\begin{tikzpicture}
    \node (0) at (0,0) {$0$};
    \node (17) at (-1,1) {$1$};
    \node (34) at (-1,2) {$2$};
    \node (51) at (-1,3) {$3$};
    \node (60) at (0,1) {$4$};
    \node (69) at (1,1) {$5$};
    \node (78) at (2,1) {$6$};
    \node (95) at (1.5,2) {$7$};
    \draw (0) -- (17) -- (34) -- (51);
    \draw (0) -- (60);
    \draw (0) -- (69);
    \draw (0) -- (78);
    \draw (78) -- (95);
    \draw (95) -- (17);
\end{tikzpicture}
\caption{${\mathcal P}(H)$}
\end{subfigure}
\begin{subfigure}[b]{0.45\textwidth}
\centering
\begin{tikzpicture}
    \node (G0) at (7,0) {$0$};
    \node (G17) at (6,1) {$1$};
    \node (G34) at (6,2) {$2$};
    \node (G51) at (6,3) {$3$};
    \node (68) at (6,4) {$4$};
    \node (53) at (7,1) {$5$};
    \node (62) at (8,1) {$6$};
    \node (55) at (9,1) {$7$};
    \draw (G0) -- (G17) -- (G34) -- (G51) -- (68);
    \draw (G0) -- (53);
    \draw (G0) -- (62);
    \draw (G0) -- (55);
\end{tikzpicture}
\caption{${\mathcal P}(G)$}
\end{subfigure}
\caption{Apéry posets in Example \ref{ex:Apery}.}
\label{fig:Aperyposets}
\end{figure}
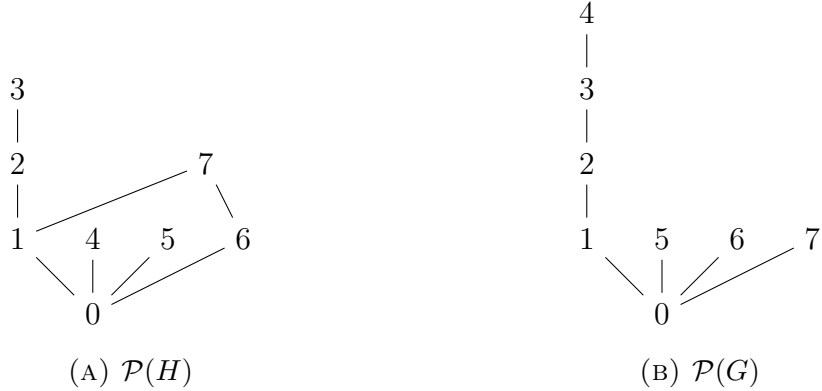

Now consider a semigroup $H \in KW(p,q)$ and let us see how its Ap\'ery poset look like.
For $0 \leq \lambda < p-y_1$ set $i_{0,\lambda}$ the label of $\lambda q$ in $\mathcal{P}(H)$, i.e. $0\leq i_{0,\lambda} < p$ and $i_{0,\lambda} \equiv \lambda q \pmod{p}$. 
Similarly, for each $1\leq j \leq n-2$ and $0\leq \lambda < y_j-y_{j+1}$, let $i_{j,\lambda}$ be the label of $h_j+\lambda q$ in $\mathcal{P}(H)$.

\begin{proposition} \label{prop:PHminrel}
Let $3 \leq p < q$ be coprime and $H \in KW(p,q)$.
The covering relations in ${\mathcal P}(H)$ are exactly the following three types:
\begin{itemize}
  \item[(T1)] $i_{j,\lambda} \precdot i_{j,\lambda+1}$, for every
    $0 \leq j \leq n-2$ and $0 \leq \lambda < y_j-y_{j+1}-1$ (setting $y_0:=p$), corresponding to adding the
    generator $q$.
  \item[(T2)] $i_{0,\lambda} \precdot i_{j,\lambda}$, for every
    $1 \leq j \leq n-2$ and $0 \leq \lambda < y_j - y_{j+1}$, corresponding
    to adding the generator $h_j$.
  \item[(T3)] $i_{n-2,\lambda} \precdot i_{0,\,\lambda+(p-2y_{n-2})}$,
    for $0 \leq \lambda < 2y_{n-2} - y_1$, corresponding to adding the
    generator $h_{n-2}$ via the relation \eqref{eq:star}.
    This type occurs if and only if $q$ is even, $x_{n-2} = q/2$, and
    $y_1 < 2y_{n-2}$.
\end{itemize}
Moreover, ${\mathcal P}(H)$ is graded if and only if T3 relations do not occur and, when it is graded, the rank function is $\rho: {\rm Ap}(H) \rightarrow \N$, $\rho(i_{0,k}) = k$, and $\rho(i_{j,k}) = k+1$ for $j\geq 1$.
\end{proposition}

\begin{proof}
By Proposition \ref{prop:cover}, $i_{j_1,\lambda_1} \precdot i_{j_2,\lambda_2}$ if and only if the difference of the corresponding elements in $\ap(H)$ is a minimal generator of $H$.
Since any two distinct Ap\'ery elements lie in distinct residue classes modulo $p$, their difference is never divisible by $p$, so we only need to consider when this minimal generator is either $q$ or $h_k$ for some $1\leq k \leq n-2$.

Let us first show that T1 and T2 are covering relations. The differences $[h_j+(\lambda+1)q] - [ h_j+\lambda q] = q$ and $[h_j+\lambda q] - \lambda q = h_j$ are minimal generators. Hence, T1 and T2 are covering relations by Proposition \ref{prop:cover}.

Now assume $q$ is even, $x_{n-2} = q/2$, and $y_1 < 2y_{n-2}$, and let us see that T3 are covering relations. We verify that both endpoints of the
claimed covering lie in $\mathrm{Ap}(H)$:
\begin{itemize}
  \item Since $y_1 > y_{n-2}$, we have $\lambda < 2y_{n-2} - y_1 < y_{n-2}$, and  $h_{n-2} + \lambda q \in \mathrm{Ap}(H)$.
  
  \item $\lambda < 2y_{n-2} - y_1$, i.e., $\lambda + p - 2y_{n-2} < p - y_1$,  so $(\lambda + p - 2y_{n-2})q \in \mathrm{Ap}(H)$.
\end{itemize}
The difference between these two elements is $(\lambda + p - 2y_{n-2})q - (h_{n-2} + \lambda q) = (p-2y_{n-2})q - h_{n-2} = h_{n-2}$, by the relation \eqref{eq:star}, which is a minimal generator, so $i_{n-2,\lambda} \precdot i_{0,\,\lambda+(p-2y_{n-2})}$ by Proposition \ref{prop:cover}.

When $y_1 \geq 2y_{n-2}$ the range $0 \leq \lambda < 2y_{n-2} - y_1$ is
empty, so T3 relations do not arise. The fact that T3 relations cannot occur when
$x_{n-2} < q/2$ will follow from Case 3.2 below.

Now, let us see that all the covering relations are of type T1, T2, or T3. Take $i_{j_1,\lambda_1},i_{j_2,\lambda_2} \in \Z_p$ and assume that $i_{j_1,\lambda_1} \precdot i_{j_2,\lambda_2}$. We consider 4 different cases depending on whether $j_1=0$ or $j_1 \geq 1$, and $j_2=0$ or $j_2\geq 1$:
\begin{enumerate}[1.]
    \item $\lambda_2 q-\lambda_1 q = m$, a minimal generator of $H$.
    \item $[h_{j_2}+\lambda_2 q] - \lambda_1q = m$, a minimal generator of $H$.
    \item $\lambda_2q-[h_{j_1}+\lambda_1q] = m$, a minimal generator of $H$.
    \item $[h_{j_2}+\lambda_2q]-[h_{j_1}+\lambda_1q] = m$, a minimal generator of $H$.
\end{enumerate}

\noindent \underline{Case 1}: $(\lambda_2-\lambda_1)q = m$.

\begin{enumerate}[{1.}1.]
    \item $m=q$: Then $\lambda_2 = \lambda_1 + 1$. This is T1.
    
    \item $m=h_k$: Then $(\lambda_2 - \lambda_1)q = pq - x_kp - y_kq$, so $x_kp = (p - y_k - \lambda_2 + \lambda_1)q$. Since $\gcd(p,q)=1$, $q$ divides $x_k$. But $x_k \in [1, q/2]$, so this is impossible.
\end{enumerate}

\noindent \underline{Case 2}: $(pq - x_{j_2}p - y_{j_2}q + \lambda_2 q) - \lambda_1 q = m$.

\begin{enumerate}[2.1.]
    \item $m=q$: Then $pq - x_{j_2}p - y_{j_2}q = (\lambda_1 - \lambda_2 + 1)q$, so $x_{j_2}p = (p - y_{j_2} - \lambda_1 + \lambda_2 - 1)q$. Since $\gcd(p,q)=1$, we get $q \mid x_{j_2}$. But $x_{j_2} \in [1, q/2]$, so this is impossible.
    
    \item $m=h_k$: Rearranging gives $(x_k - x_{j_2})p + (y_k - y_{j_2} + \lambda_2 - \lambda_1)q = 0$. Reducing modulo $q$: $q \mid (x_{j_2} - x_k)$. Since $x_{j_2}, x_k \in [1, q/2]$, we have $x_{j_2} = x_k$, hence $j_2 = k$, and substituting back gives $\lambda_2 = \lambda_1$. This is T2.
\end{enumerate}

\noindent \underline{Case 3}: $\lambda_2 q - (pq - x_{j_1}p - y_{j_1}q + \lambda_1 q) = m$.

\begin{enumerate}[3.1.]
    \item $m=q$: Then $pq - x_{j_1}p - y_{j_1}q = (\lambda_2 - \lambda_1 - 1)q$, so $x_{j_1}p = (p - y_{j_1} - \lambda_2 + \lambda_1 + 1)q$. Since $\gcd(p,q)=1$, we get $q \mid x_{j_1}$. But $x_{j_1} \in [1, q/2]$, so this is impossible.
    
    \item $m=h_k$: Rearranging gives $(x_{j_1} + x_k)p = (2p - y_{j_1} - y_k + \lambda_1 - \lambda_2)q$. Since $\gcd(p,q)=1$, we get $q \mid (x_{j_1} + x_k)$. Since $x_{j_1}, x_k \in [1, q/2]$, we have $x_{j_1} + x_k \in [2, q]$, so $x_{j_1} + x_k = q$, forcing $x_{j_1} = x_k = q/2$. Therefore, $j_1 = k = n-2$, with $q$ necessarily even. Substituting back:
    \[(\lambda_2 - \lambda_1)q = 2h_{n-2} = (p - 2y_{n-2})q,\]
    giving $\lambda_2 = \lambda_1 + (p - 2y_{n-2})$. Since $\lambda_2 q \in \operatorname{Ap}(H)$, we have $\lambda_2 < p - y_1$, and hence $0 \leq \lambda_1 < 2y_{n-2} - y_1$, so $y_1 < 2y_{n-2}$. This is T3.

    In particular, if $x_{n-2} < q/2$ then $1 < x_{j_1} + x_k < q$ for all
    $j_1, k$, so this case never applies and T3 relations cannot occur.
\end{enumerate}

\noindent \underline{Case 4}: $(pq-x_{j_2}p-y_{j_2}q + \lambda_2 q)-(pq-x_{j_1}p-y_{j_1}q + \lambda_1q) = m$.

\begin{enumerate}[4.1.]
    \item $m=q$: The equation gives $(x_{j_1} - x_{j_2})p + (y_{j_1} - y_{j_2} + \lambda_2 - \lambda_1 - 1)q = 0$. Reducing modulo $q$, we get $q \mid (x_{j_1} - x_{j_2})$. Since $x_{j_1}, x_{j_2} \in [1, q/2]$, we have $x_{j_1} = x_{j_2}$, i.e., $j_1 = j_2$, and then $\lambda_2 = \lambda_1 + 1$. This is T1.
    
    \item $m=h_k$: The equation gives
    \[(x_{j_1} + x_k - x_{j_2})p + (y_{j_1} + y_k - y_{j_2} + \lambda_2 - \lambda_1)q = pq.\]
    Reducing modulo $q$, $q \mid (x_{j_1} + x_k - x_{j_2})$. Since $x_{j_1}, x_k \in [1, q/2]$ and $x_{j_2} \in [1, q/2]$, the number $x_{j_1} + x_k - x_{j_2}$ lies in $[2 - q/2,\, q-1]$, and the only multiple of $q$ in this range is $0$, giving $x_{j_2} = x_{j_1} + x_k$. The equation then gives
    \[ \lambda_2 = \lambda_1 + p + y_{j_2} - y_{j_1} - y_k.\]
    For $h_{j_2}+\lambda_2q \in \ap(H)$ we need $\lambda_2 < y_{j_2}$, which together with $\lambda_1 \geq 0$ forces $p < y_{j_1} + y_k$. But $y_{j_1}, y_k \leq \lfloor p/2 \rfloor$, so $y_{j_1} + y_k \leq 2\lfloor p/2 \rfloor \leq p$, a contradiction.
\end{enumerate}

This exhausts all cases, completing the proof that T1, T2, and T3 are all the covering relations in $\mathcal{P}(H)$.

When T3 relations are present, i.e., when $q = 2x_{n-2}$ and $y_1 < 2y_{n-2}$, the elements $(\lambda+p-2y_{n-2})q = 2h_{n-2} + \lambda q \in \ap(H)$ can be written using $\lambda+p-2y_{n-2}$ and $2+\lambda$ minimal generators, respectively. Note that $p-2y_{n-2} \neq 2$ because $p$ is odd when $q$ is even.
\end{proof}

\begin{definition}
For a semigroup $H \in KW(p,q)$, define the boolean variable 
\[\varepsilon_H = \begin{cases}
    1 & \text{if } q \text{ is even, } x_{n-2} = q/2 \text{, and } y_1 <2y_{n-2}, \\
    0 & \text{otherwise.}
\end{cases}\]
\end{definition}

\begin{figure}[htbp]
\begin{subfigure}[b]{0.495\textwidth}
\centering
\includegraphics[width=0.95\textwidth]{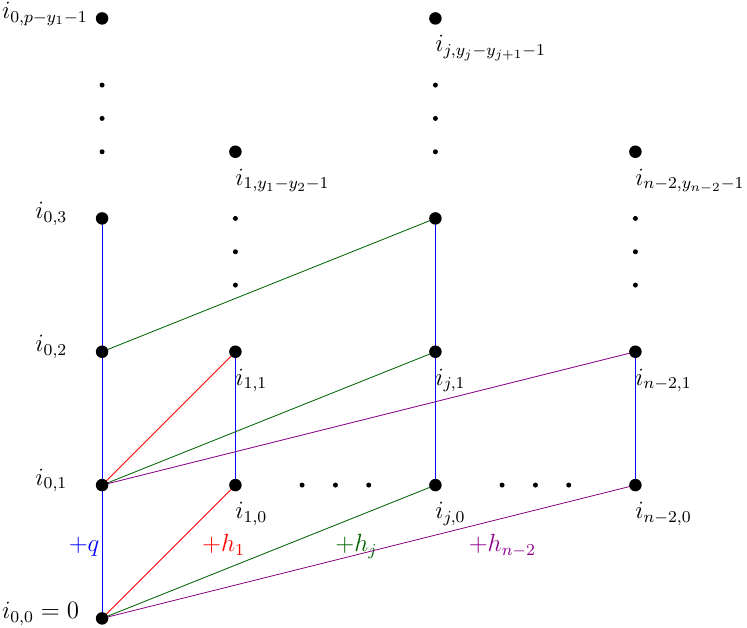}
\caption{Case $\varepsilon_H = 0$.}
\end{subfigure}
\begin{subfigure}[b]{0.495\textwidth}
\centering
\includegraphics[width=0.95\textwidth]{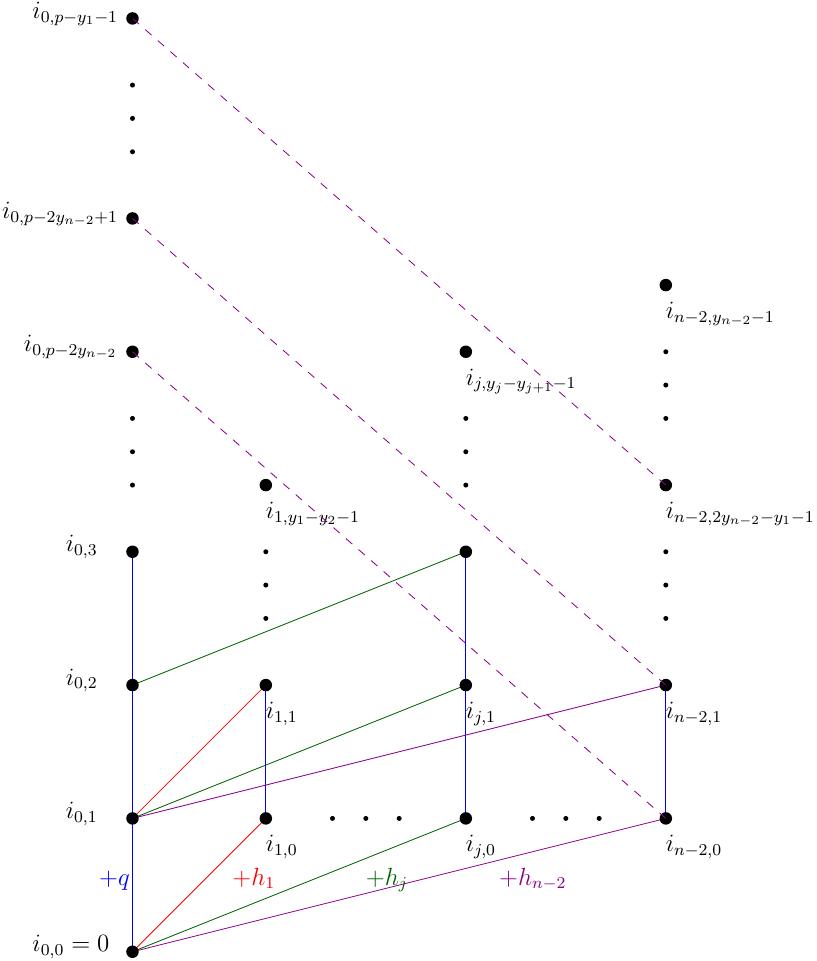}
\caption{Case $\varepsilon_H = 1$. The extra relations T3 making ${\mathcal P}(H)$ non graded are dashed.}
\end{subfigure}
\caption{Hasse diagram of the Ap\'ery poset of a semigroup $H \in KW(p,q)$.}
\label{fig:hasse}
\end{figure}

Let $H, G \in KW(p,q)$ be of the same embedding dimension $n$. Write $H= \la p,q, h_1,...,h_{n-2} \ra$ and $G= \la p,q, g_1,...,g_{n-2} \ra$. Further, write $h_i = pq-x_ip-y_iq$ and $g_i = pq - z_ip-w_iq$, $1 \leq i \leq n-2$.

\begin{proposition}\label{prop:KWfaces}
    Let $H, G \in KW(p,q)$. Then $H$ and $G$ belong to the interior of the same face of the Kunz cone $\mC_p$ if and only if
    \begin{enumerate}
        \item $e(H)=e(G)=n$,
        \item $y_i = w_i$ for all $1 \leq i \leq n-2$, and
        \item $\varepsilon_H = \varepsilon_G$.
    \end{enumerate}
\end{proposition}

\begin{proof}
By Proposition \ref{prop:PHminrel}, the Apéry poset of any semigroup $H \in KW(p,q)$ is completely determined by the embedding dimension of $H$, $n$, the sequence $p/2 \geq y_1>\dots>y_{n-2}>0$, and the boolean variable $\varepsilon_H$. Thus, the result follows from Theorem \ref{thm:KunzApery}.
\end{proof}

\begin{corollary}\label{cor:Kunzcone}
    Let $H\in KW(p,q)$ of embedding dimension $n$ be such that $y_i =(n-i-1)y$ for some $y \in \N$ with $(n-2)y \leq p/2$. Then the Betti numbers of $\k[H]$ are \[\beta_i = i {n \choose i+1}, \ 1 \leq i \leq n-1 \, .\]
\end{corollary}
\begin{proof}
Since $n \geq 4$, we have $y_1 = (n-2)y \geq 2y =2 y_{n-2}$, so $\varepsilon_H = 0$.
Consider the semigroup $H' \in \kwpqd$ defined by the sequence $x_1 = 1<x_2=2<\dots < x_{n-2} = n-2$ and the sequence $y_1>y_2>\dots>y_{n-2}$, $y_i = (n-1-i)y$. By Proposition \ref{prop:KWfaces}, $H$ and $H'$ lie in the interior of the same face of the Kunz cone, since they share the sequence $y_1>y_2>\dots>y_{n-2}$ and $\varepsilon_H = \varepsilon_{H'} = 0$. Therefore, $\k[H]$ and $\k[H']$ have the same Betti sequence by Theorem 2.7 of [1], that is, $\beta_i = i {n \choose i+1}$, $1\leq i \leq n-1$ by the construction of the Eagon-Northcott complex for $I_{H'}$.
\end{proof}

\begin{example}\label{ex:Aperysameface}
    In Example \ref{ex:Apery}, $G$ is not in $KW_D(8,17)$ as $53,62,55$ are not in an arithmetic sequence. However, $G$ is in the same face of $\mC_8$ as $G'= \la 8,17,69,70,71 \ra$. Note that for both $G$ and $G'$, $y_1=3,y_2=2,y_3=1$. Their Betti sequence is $4,15,20,10,1$, and following is the Hasse diagram of $\mathcal P(G)$ and $\mathcal P(G')$:
    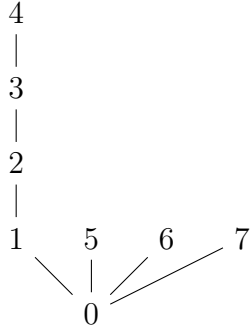
\begin{figure}[htbp]
\centering
\begin{tikzpicture}
    \node (G0) at (7,0) {$0$};
    \node (G17) at (6,1) {$1$};
    \node (G34) at (6,2) {$2$};
    \node (G51) at (6,3) {$3$};
    \node (68) at (6,4) {$4$};
    \node (53) at (7,1) {$5$};
    \node (62) at (8,1) {$6$};
    \node (55) at (9,1) {$7$};
    \draw (G0) -- (G17) -- (G34) -- (G51) -- (68);
    \draw (G0) -- (53);
    \draw (G0) -- (62);
    \draw (G0) -- (55);
\end{tikzpicture}
\caption{${\mathcal P}(G)$ and ${\mathcal P}(G')$}
\label{fig:Aperyposetssameface}
\end{figure}
\end{example}

\begin{remark}
By Corollary \ref{cor:Kunzcone}, the number of numerical semigroups in $KW(p,q)$ whose Betti numbers are $\beta_i = i {n\choose i+1}$, $1\leq i \leq n-1$, is at least $\sum_{n=1}^{p'} \left\lfloor \frac{p'}{n} \right\rfloor {q' \choose n}$, where $p' = \lfloor p/2 \rfloor$ and $q' = \lfloor q/2 \rfloor$.
\end{remark}

It is interesting to note that even if $H\in KW(p,q)$ does not satisfy the hypothesis of Corollary \ref{cor:Kunzcone}, the conclusion still seems to hold. In particular, if the $y_1,...,y_{n-2}$ are not in an arithmetic sequence, is it still true that $\beta_i = i {n \choose i+1}$ for $1 \leq i \leq n-1$? We give an example to support a positive answer:
\begin{example}
    $H= \la 8,17, 36,45,63 \ra \in KW(8,17) \bs KW_D(8,17)$. The Betti sequence of $\k [H]$ is $4,15,20,10,1$, and yet, $H$ is not even in the same face as some $H' \in KW_D(8,17)$.
\end{example}

This is just one of many examples but we have yet to find a KW-semigroup of embedding dimension $n-1$, whose Betti numbers are not $\beta_i = i {n \choose i+1}$, $1\leq i \leq n-1$. This is clearly true for embedding dimension 3.  With this in mind, we end this article with the following conjecture:

\begin{conjecture}
    The Betti numbers of any numerical semigroup $H \in KW(p,q)$ of embedding dimension $n$ are   $\beta_i = i {n \choose i+1}$, $1\leq i \leq n-1$.
\end{conjecture}

\end{document}